\pdfoutput=1
\documentclass[a4paper,12pt]{amsart}
\usepackage{amsmath,amssymb,amsxtra,amscd,mathrsfs,enumitem}
\usepackage[margin=3cm]{geometry} 
\usepackage{tikz} 
\usepackage{color} 
\usepackage{booktabs}
\usepackage{longtable}  
\usepackage{calc} 
\usepackage{hyperref} 
\usepackage[all]{xy}

\newtheorem{theorem}{Theorem} 
\newtheorem{lemma}[theorem]{Lemma} 
\newtheorem{proposition}[theorem]{Proposition} 
\newtheorem{corollary}[theorem]{Corollary} 
 
\newtheorem*{mainresult}{Main Results} 

\theoremstyle{definition}

\newtheorem{remark}[theorem]{Remark}

\usepackage{xspace} 

\newcommand{\C}{\mathbb{C}}
\newcommand{\Z}{\mathbb{Z}} 
\newcommand{\N}{\mathbb{N}}
\newcommand{\Q}{\mathbb{Q}} 
\newcommand{\R}{\mathbb{R}}

\newcommand{\PP}{\mathbb{P}} 

\newcommand{\cO}{\mathcal{O}}

\newcommand{\cN}{\mathcal{N}}

\newcommand{\cC}{\mathcal{C}} 

\newcommand{\hD}{{\widehat{D}}}
\newcommand{\hvarphi}{\hat{\varphi}}

\newcommand{\hjmath}{{\hat{\jmath}}}
\newcommand{\hpi}{\hat{\pi}}
\newcommand{\hp}{{\hat{p}}}

\newcommand{\tX}{{\widetilde{X}}}
\newcommand{\td}{{\tilde{d}}}
\newcommand{\tQ}{{\widetilde{Q}}}
 
\newcommand{\tDelta}{{\widetilde{\Delta}}} 
\newcommand{\ttau}{{\tilde{\tau}}}

\newcommand{\bbf}{\mathbf{f}}

\newcommand{\btau}{\boldsymbol{\tau}} 
\newcommand{\tbtau}{\tilde{\btau}}

\newcommand{\bvarsigma}{\boldsymbol{\varsigma}} 
\newcommand{\ovbvarsigma}{\overline{\bvarsigma}}

\newcommand{\sfF}{\mathsf{F}}
\newcommand{\sfH}{\mathsf{H}}

\newcommand{\frq}{\mathfrak{q}}  
\newcommand{\frs}{\mathfrak{s}}

\newcommand{\frt}{\mathfrak{t}}

\newcommand{\scrF}{\mathscr{F}}

\newcommand{\NE}{\operatorname{NE}} 

\newcommand{\NEN}{\NE_{\N}}
\newcommand{\id}{\operatorname{id}} 
 
\newcommand{\pt}{\operatorname{pt}} 
\newcommand{\GL}{\operatorname{GL}} 
\newcommand{\Hom}{\operatorname{Hom}} 
\newcommand{\Aut}{\operatorname{Aut}} 
\newcommand{\End}{\operatorname{End}} 

\newcommand{\ch}{\operatorname{ch}} 
\newcommand{\QDM}{\operatorname{QDM}}

\newcommand{\pr}{\operatorname{pr}} 
\newcommand{\FT}{\operatorname{FT}}

\newcommand{\dec}{\operatorname{\mathsf{dec}}} 
\newcommand{\Hod}{\operatorname{\mathsf{Hod}}}

\newcommand{\iu}{\mathtt{i}}
\newcommand{\ttt}{\mathtt{t}} 
\newcommand{\st}{{\rm s}}

\newcommand{\bPsi}{\boldsymbol{\Psi}}

\def\corr#1{\left\langle#1 \right\rangle}

\begin{document} 

\title{Notes on the decomposition theorem for blowups} 
\author{Hiroshi Iritani}
\email{iritani@math.kyoto-u.ac.jp} 
\address{Department of Mathematics, Graduate School of Science, 
Kyoto University, Kitashirakawa-Oiwake-cho, Sakyo-ku, 
Kyoto, 606-8502, Japan}


\begin{abstract} 
We discuss arithmetic and Hodge-theoretic properties of the isomorphisms appearing in the decomposition theorem \cite{Iritani:monoidal} for quantum cohomology of blowups. These properties underpin the application to the rationality questions by Katzarkov-Kontsevich-Pantev-Yu \cite{KKPY:birational}. 
\end{abstract} 
\maketitle 


Let $X$ be a smooth projective variety, and let $Z\subset X$ be a smooth subvariety of codimension $r\ge 2$. Let $\tX$ be the blowup of $X$ along $Z$. The \emph{decomposition theorem} \cite[Theorems 1.1, 5.18]{Iritani:monoidal} for blowups establishes the following isomorphism of (formal) quantum $D$-modules:
\begin{equation} 
\label{eq:Psi}
\Psi \colon \QDM(\tX)^{\rm la} \cong \tau^*\QDM(X)^{\rm la} \oplus \bigoplus_{j=0}^{r-2} \varsigma_j^*\QDM(Z)^{\rm la} 
\end{equation} 
for some formal maps $\tau \colon H^*(\tX) \to H^*(X)$ and $\varsigma_j\colon H^*(\tX) \to H^*(Z)$ that induce a formal isomorphism at $\ttau=0$ 
\begin{equation} 
\label{eq:change_of_variables}
H^*(\tX) \to H^*(X) \oplus H^*(Z)^{\oplus (r-1)}, \quad \ttau \mapsto (\tau(\ttau), \{\varsigma_j(\ttau)\}_{j=0}^{r-2}). 
\end{equation} 
These maps $\Psi$, $\tau(\ttau)$ and $\varsigma_j(\ttau)$ are defined over $\C$ in \cite{Iritani:monoidal}. In these notes we show the following: 
\begin{mainresult}[see Propositions \ref{prop:cyclotomic}, \ref{prop:Hodge-equivariance} and Corollary \ref{cor:Hodge-fixed} for precise statements]   
The maps $\tau(\ttau)$, $\varsigma_j(\ttau)$, and $\Psi$ are essentially defined over a cyclotomic field. Furthermore, the change-of-variables map \eqref{eq:change_of_variables} restricts to a formal isomorphism between the subspaces of complexified Hodge classes. When $\ttau$ is a complexified Hodge class, the isomorphism $\Psi|_\ttau$ restricts to an isomorphism between these Hodge subspaces. 
\end{mainresult} 


\begin{remark} 
We summarize the notation from \cite{Iritani:monoidal} necessary to state the results in these notes. In the proofs of these results, however, we shall freely adopt the notation and conventions of \cite{Iritani:monoidal}. 
\end{remark} 

\subsection*{Acknowledgements}
I thank Sergey Galkin for crucial clarifications, and Alessio Corti, J\'er\'emy Gu\'{e}r\'{e}, Yank{\i} Lekili, and Emanuele Macr\`i for their questions, which motivated me to write these notes.

\section{Recollection of notation from \cite{Iritani:monoidal}}
We recall notation necessary to state the results. Every module in these notes is $\Z$-graded and its completion will always be considered in the graded sense. For a $\Z$-graded module $N = \bigoplus_k N^k$ and a descending chain $\{I_n\}_{n\ge 1}$ of homogeneous submodules $I_n = \bigoplus_k I_n^k\subset N$, the graded completion of $N$ with respect to $\{I_n\}$ is defined to be $\widehat{N} = \bigoplus \widehat{N}^k$, where $\widehat{N}^k = \varprojlim_n N^k/I_n^k$. 

%
Let $\NEN(X) \subset H_2(X,\Z)$ denote the monoid of effective curve classes in $X$. We denote by $Q^d\in \C[\NEN(X)]$ the element corresponding to $d\in \NEN(X)$ in the monoid ring, and define its degree by $\deg Q^d = 2 c_1(X) \cdot d$. The variable $Q$ is called the \emph{Novikov variable}. Let $\omega\in H^2(X,\R)$ be an ample class. The \emph{Novikov ring} $\C[\![Q]\!]$ is defined to be the graded completion of $\C[\NEN(X)]$ with respect to the descending chain $I_n = \langle Q^d : \omega \cdot d\ge n\rangle_\C \subset \C[\NEN(X)]$ of submodules. The completion does not depend on the choice of $\omega$. Similarly, for any $\Z$-graded ring $K$, we define $K[\![Q]\!]$ to be the graded completion of $K[\NEN(X)]$. 

For a variable $x$ of degree $\deg x \in \Z$ and a $\Z$-graded ring $K$, we define the ring $K(\!(x)\!)$ of formal Laurent series to be the graded completion of $K[x,x^{-1}]$ with respect to the descending chain $I_n = x^n K[x]$ of submodules. 

For $Y= X$, $\tX$, or $Z$, let $H^*(Y) = H^*(Y;\C)$ denote the cohomology group with complex coefficients and let $\{\phi_{Y,i}\}_i$ denote a homogeneous basis of $H^*(Y;\Q)$ over $\Q$. Let $\tau$, $\ttau$ and $\sigma$ denote the parameters of the quantum cohomology of $X$, $\tX$ and $Z$ respectively; they take values in $H^*(X)$, $H^*(\tX)$ and $H^*(Z)$. We expand these parameters as $\tau = \sum_i \tau^i \phi_{X,i}$, $\ttau = \sum_i \ttau^i \phi_{\tX,i}$ and $\sigma = \sum_i \sigma^i \phi_{Z,i}$. The variables $\tau^i$, $\ttau^i$, $\sigma^i$ are assigned the degrees 
\[
\deg \tau^i = 2- \deg \phi_{X,i}, \quad \deg \ttau^i = 2 - \deg \phi_{\tX,i}, \quad \deg \sigma^i = 2 - \deg \phi_{Z,i}.  
\]
They are also assigned the parities $|\tau^i|, |\ttau^i|, |\sigma^i|\in \Z/2\Z$ that are congruent modulo 2 to their degrees. We require that they are supercommutative with respect to their parities, e.g.~$\tau^i \tau^j = (-1)^{|\tau^i| |\tau^j|} \tau^j \tau^i$. 

For a $\Z$-graded ring $K$, $K[\![\ttau]\!]$ (or $K[\![\tau]\!]$, $K[\![\sigma]\!]$) denotes the graded completion of the polynomial ring $K[\ttau]=K[\{\ttau^i\}]$ (resp.~$K[\tau]$, $K[\sigma]$) with respect to the polynomial degree. 
Due to the supercommutativity, we have 
\[
K[\![\ttau]\!] = K[\![\ttau^i: \text{even}]\!] \otimes_K \bigwedge^\bullet_K \left(\bigoplus_{\ttau^i: \text{odd}} K\ttau^i \right). 
\]
We also use the shorthand notation such as $K[\![Q,\ttau]\!] := K[\![Q]\!][\![\ttau]\!]$ (and similarly for $K[\![Q,\tau]\!]$ and $K[\![Q,\sigma]\!]$). 

Let $\varphi \colon \tX \to X$ denote the blowup morphism along $Z \subset X$ and let $D\subset \tX$ be the exceptional locus. We have $D\cong \PP(\cN_{Z/X})$. These spaces fit into the following commutative diagram:  
\begin{equation} 
\label{eq:diag} 
\begin{aligned} 
\xymatrix{
D \ar[d]_{\pi} \ar@{^{(}->}[r]^{\jmath} & \tX \ar[d]^{\varphi} \\ 
Z \ar@{^{(}->}[r]^{\imath} & X  
}
\end{aligned} 
\end{equation} 
where $\imath, \jmath$ are the natural inclusions and $\pi$ is the natural projection. We denote by $p\in H^2(D;\Z)$ the relative hyperplane class $c_1(\cO_D(1))$ over $D = \PP(\cN_{Z/X})$. Given bases $\{\phi_{X,i}\}$ and $\{\phi_{Z,k}\}$ of $H^*(X;\Q)$ and $H^*(Z;\Q)$, respectively, the set $\{\varphi^*\phi_{X,i}, \jmath_*(p^l\pi^*\phi_{Z,k})\}_{i,k,0\le l\le r-2}$ forms a basis of $H^*(\tX;\Q)$. 

Let $z$ be a variable of degree two and let $\frq$ denote the Novikov variable associated with the class of a line contracted under the map $\tX \to X$. We have $\deg \frq = 2(r-1)$. We also set 
\[
\frs = \begin{cases} 
r-1 & \text{if $r$ is even;}\\ 
2(r-1) & \text{if $r$ is odd.} 
\end{cases}
\]
The quantum $D$-modules of $X$, $\tX$ and $Z$ are originally defined over the respective Novikov rings. Let $Q$, $\tQ$, $Q_Z$ denote the Novikov variables for $X$, $\tX$ and $Z$ respectively. In \cite[(1.1)]{Iritani:monoidal}, we embed the Novikov rings of $X$, $\tX$ and $Z$ into the common ring $\C(\!(\frq^{-1/\frs})\!)[\![Q]\!]$ as follows: 
\begin{equation}
\label{eq:embed} 
\begin{aligned}
\C[\![Q]\!] & \hookrightarrow \C(\!(\frq^{-1/\frs})\!)[\![Q]\!] \qquad && \text{in an obvious way} \\ 
\C[\![\tQ]\!] & \hookrightarrow \C(\!(\frq^{-1/\frs})\!)[\![Q]\!] && \tQ^\td \mapsto Q^{\varphi_*\td} \frq^{-D\cdot \td} \\ 
\C[\![Q_Z]\!] & \to \C(\!(\frq^{-1/\frs})\!)[\![Q]\!] && Q_Z^d \mapsto Q^{\imath_*d}\frq^{-\rho_Z\cdot d/(r-1)}
\end{aligned}
\end{equation}  
where $\rho_Z = c_1(\cN_{Z/X})$, and defined the following Laurent forms via the base change to this common ring (see \cite[(5.37)]{Iritani:monoidal}):  
\begin{align*}
\QDM(X)^{\rm la} & := H^*(X)\otimes \C[z](\!(\frq^{-1/\frs})\!)[\![Q,\tau]\!] \\
\QDM(\tX)^{\rm la} & := H^*(\tX) \otimes \C[z](\!(\frq^{-1/\frs})\!)[\![Q,\ttau]\!] \\ 
\QDM(Z)^{\rm la} & := H^*(Z) \otimes \C[z](\!(\frq^{-1/\frs})\!)[\![Q,\sigma]\!].  
\end{align*} 
The pull-backs $\tau^*\QDM(X)^{\rm la}$, $\varsigma_j^* \QDM(Z)^{\rm la}$ are defined as modules as: 
\begin{align*} 
\tau^*\QDM(X)^{\rm la} & := H^*(X)\otimes \C[z](\!(\frq^{-1/\frs})\!)[\![Q,\ttau]\!]\\ 
\varsigma_j^* \QDM(Z)^{\rm la} & := H^*(Z)\otimes\C[z](\!(\frq^{-1/\frs})\!)[\![Q,\ttau]\!] 
\end{align*} 
They are equipped with the appropriate pull-back connections (see \cite[(5.41)--(5.43)]{Iritani:monoidal}). 

\begin{remark} 
All variables, as well as cohomology classes, carry both a $\Z$-degree and a $\Z/2\Z$-parity. Except for the parameters $\tau^i$, $\ttau^i$ and $\sigma^i$ associated with odd cohomology classes, all other variables ($Q$, $z$, $\frq^{1/\frs}$, etc.) have even parity. Note that the parity is not always congruent modulo $2$ to the degree: we have $\deg \frq^{1/\frs} =1$ when $r$ is odd, but the parity of $\frq^{1/\frs}$ remains even. 
\end{remark} 

\section{Arithmetic properties}
\label{sec:arithmetic} 
First we recall several properties of $\tau(\ttau)$, $\varsigma_j(\ttau)$ and $\Psi$ established in \cite{Iritani:monoidal}. 
The formal changes of variables $\tau= \tau(\ttau)$, $\sigma = \varsigma_j(\ttau)$ from \eqref{eq:change_of_variables} are expanded in the form: 
\[
\tau(\ttau) = \sum_i \tau^i(\ttau) \phi_{X,i}, \quad \varsigma_j(\ttau) = \sum_i \varsigma_j^i(\ttau) \phi_{Z,i}  
\]
with $\tau^i(\tau) \in \C(\!(\frq^{-1})\!)[\![Q,\ttau]\!]$ and $\varsigma_j^i(\ttau) \in \C(\!(\frq^{-\frac{1}{r-1}})\!)[\![Q,\ttau]\!]$. Note the differing exponents of $\frq$; note also that $\C(\!(\frq^{-1})\!)=\C[\frq^\pm]$, $\C(\!(\frq^{-\frac{1}{r-1}})\!)=\C[\frq^{\pm\frac{1}{r-1}}]$ by the convention on completion. 
They satisfy the following properties (see \cite[Theorem 5.18]{Iritani:monoidal}): 
\begin{enumerate}[label=(\alph*)]
\item \label{item:homogeneity_coordinatechange}
The formal power series $\tau^i(\ttau)$, $\varsigma_j^i(\ttau)$ have the same degrees and the parities as the variables $\tau^i$, $\sigma^i$. 
\item \label{item:asymp_coordinatechange} 
We have 
\begin{align*} 
\tau(\ttau)|_{Q=\ttau=0} & = \frq^{-1}[Z] + O(\frq^{-2}) \in H^*(X)\otimes \C[\frq^{-1}]\\
\varsigma_j(\ttau)|_{Q=\ttau=0} & = -(r-1) \lambda_j + h_{Z,j} + O(\frq^{-\frac{1}{r-1}}) \in H^*(Z)\otimes \C[\frq^{\pm \frac{1}{r-1}}] 
\end{align*} 
where $\lambda_j = -e^{-\frac{2\pi\iu}{r-1}j} (e^{-\pi\iu}\frq)^{\frac{1}{r-1}}$ and $h_{Z,j} = \frac{2\pi\iu}{r-1} (j+\frac{1}{2}) \rho_Z$ with $\rho_Z = c_1(\cN_{Z/X})$. 
\item The Jacobian matrix of the map \eqref{eq:change_of_variables} is invertible over $\C(\!(\frq^{-\frac{1}{r-1}})\!)[\![Q,\ttau]\!]$; hence, \eqref{eq:change_of_variables} defines a formal isomorphism at $\ttau=0$. 
\end{enumerate} 
The statement about the parity in Part \ref{item:homogeneity_coordinatechange} was omitted in \cite{Iritani:monoidal}, but it is obvious from the construction.  

The decomposition isomorphism $\Psi$ from \eqref{eq:Psi} is an invertible element 
\[
\Psi \in \Hom(H^*(\tX), H^*(X)\oplus H^*(Z)^{\oplus (r-1)}) \otimes \C[z](\!(\frq^{-1/\frs})\!)[\![Q,\ttau]\!].  
\]
We write $\Psi = (\Psi_X, \Psi_{Z,0},\dots,\Psi_{Z,{r-2}})$, where $\Psi_X$ is a $\Hom(H^*(\tX),H^*(X))$-valued formal power series and each $\Psi_{Z,j}$ is a $\Hom(H^*(\tX),H^*(Z))$-valued formal power series. We have the following \cite[Theorem 5.18]{Iritani:monoidal}: 
\begin{enumerate}[label=(\alph*), start=4] 
\item \label{item:homogeneity_Psi} 
The endomorphism $\Psi_X$ is homogeneous of degree zero and each $\Psi_{Z,j}$ is homogeneous of degree $-r$. They also preserve the parity. 

\item \label{item:asymp_Psi} 
Recall the basis $\{\varphi^*\phi_{X,i}, \jmath_*(p^l\pi^*\phi_{Z,k})\}_{i,k,0\le l\le r-2}$ of $H^*(\tX;\Q)$ introduced following \eqref{eq:diag}. The following asymptotics hold: 
\begin{align*} 
\Psi_X (\varphi^* \phi_{X,i})|_{Q=\ttau=0} & = \phi_{X,i} + O(\frq^{-1}) \\  
\Psi_X( \jmath_*(p^l \pi^*\phi_{Z,k}))|_{Q=\ttau=0} & = O(\frq^{-1}) \\
\Psi_{Z,j}(\varphi^*\phi_{X,i})|_{Q=\ttau=0} & = q_{Z,j}(\phi_{X,i}|_Z + O(\frq^{-\frac{1}{r-1}})) \\
\Psi_{Z,j}(\jmath_*(p^l\pi^*\phi_{Z,k}))|_{Q=\ttau=0} & = q_{Z,j} (-1)^l \lambda_j^{l+1} (\phi_{Z,k} + O(\frq^{-\frac{1}{r-1}})) 
\end{align*} 
where $q_{Z,j} = \frac{1}{\iu \sqrt{r-1}} e^{\frac{\pi \iu r}{r-1}j} (e^{-\pi\iu} \frq)^{-\frac{r}{2(r-1)}}$. 
\end{enumerate} 
The statement about parity in Part \ref{item:homogeneity_Psi} was again omitted in \cite{Iritani:monoidal}, but it is obvious from the construction. 

\begin{proposition} 
\label{prop:cyclotomic} 
We set $\frt :=e^{-\pi\iu}\frq = -\frq$. Recall the quantities $h_{Z,j}$ from \ref{item:asymp_coordinatechange} and $q_{Z,j}$ from \ref{item:asymp_Psi}. We have 
\begin{align*} 
\tau(\ttau) & \in H^*(X;\Q) \otimes \Q(\!(\frt^{-1})\!)[\![Q,\ttau]\!] \\
\varsigma_0(\ttau) & \in  h_{Z,0} + H^*(Z;\Q) \otimes \Q(\!(\frt^{-\frac{1}{r-1}})\!)[\![Q,\ttau]\!] \\
\Psi_X & \in \Hom(H^*(\tX;\Q), H^*(X;\Q))\otimes \Q[z](\!(\frt^{-1})\!)[\![Q,\ttau]\!] \\ 
\Psi_{Z,0} & \in q_{Z,0} \Hom(H^*(\tX;\Q),H^*(Z;\Q))\otimes \Q[z](\!(\frt^{-\frac{1}{r-1}})\!)[\![Q,\ttau]\!] 
\end{align*} 
Moreover, $\varsigma_j(\ttau)$ and $\Phi_{Z,j}$ are obtained respectively from $\varsigma_0(\ttau)$ and $\Phi_{Z,0}$ by the monodromy transformation with respect to the path $\theta \mapsto e^{-2\pi \iu\theta} \frt$, $\theta \in [0,j]$: 
\begin{align*} 
\varsigma_j(\ttau)- h_{Z,j} & = \left.(\varsigma_0(\ttau)-h_{Z,0}) \right|_{\frt \to e^{-2\pi\iu j} \frt} \\ 
\Psi_{Z,j} & = \left.\Psi_{Z,0}\right|_{\frt \to e^{-2\pi\iu j}\frt} 
\end{align*} 
In particular, when written in the original coordinate $\frq$, the components of $q_{Z,j}^{-1}\Psi_{Z,j}$ and $\varsigma_j(\ttau) - h_{Z,j}$ are formal power series in $\frq$, $Q$, $\ttau$, $z$ with coefficients in the cyclotomic field $\Q(e^{\frac{\pi\iu}{r-1}}) = \Q(e^{2\pi\iu/\frs})$. 
\end{proposition}

\begin{proof} 
Recall from \cite[Sections 5]{Iritani:monoidal} that the isomorphism $\Psi$ is induced from the map $\bPsi$ that is determined by the following commutative diagram: 
\[
\xymatrix{
& \ar[ld]_{\FT_\tX\sphat}^{\cong} \QDM_T(W)_\tX\sphat \ar[rd]^{\FT_X\sphat\oplus\bigoplus_{j=0}^{r-2} \FT_{Z,j}\sphat} & \\
\tbtau^*\QDM(\tX)^{\rm ext} \ar[rr]^-{\bPsi} & & \btau^*\QDM(X)^{\rm La} \oplus \bigoplus_{j=0}^{r-2} \bvarsigma_j^*\QDM(Z)^{\rm La}
}
\]
where $\tbtau\colon H^*_T(W) \to H^*(\tX)$, $\btau \colon H^*_T(W) \to H^*(X)$, $\bvarsigma_j \colon H^*_T(W) \to H^*(Z)$ are formal maps\footnote{These are denoted by the non-bold symbols $\ttau,\tau,\varsigma_j$ in \cite[Theorem 5.9]{Iritani:monoidal}; we use bold symbols here to avoid confusion.}. The map $\tbtau$ defines a formal isomorphism when restricted to a finite-dimensional subspace $\sfH \subset H^*_T(W)$. The maps $\tau$ and $\varsigma_j$ in the present notes are given as the compositions $\btau \circ (\tbtau|_\sfH)^{-1}$ and $\bvarsigma_j\circ (\tbtau|_{\sfH})^{-1}$ respectively. The map $\Psi$ is obtained by restricting $\bPsi$ to the subspace $\sfH$. Note that $\bPsi$ appears in \cite[Theorem 5.9]{Iritani:monoidal}, and the reduction to the finite-dimensional base is discussed in \cite[Section 5.7]{Iritani:monoidal}. 

It is evident from the construction that the isomorphism $\FT\sphat_\tX$ and the change of variables $\tbtau \colon H^*_T(W) \to H^*(\tX)$ are defined over $\Q$. In fact, these data are determined by the discrete Fourier transform $\sfF_\tX(J_W(\theta))$ (via \cite[Corollary 4.11, Proposition 5.1]{Iritani:monoidal}), which is itself defined over $\Q$. Similarly, the Fourier transformation $\FT_X\sphat$ and the change of variables $\btau \colon H^*_T(W) \to H^*(X)$ are defined over $\Q$. Note also that the maps $\FT\sphat_\tX$, $\FT\sphat_X$, $\tbtau$, $\btau$ only involve integer powers of $\frq = -\frt$. 

It follows that the map $\Psi_X$ and the change of variables $\tau = \btau \circ (\tbtau|_\sfH)^{-1} \colon H^*(\tX) \to H^*(X)$ are defined over $\Q$. Thus it suffices to study the Fourier transformations $\FT\sphat_{Z,j}$ and the associated coordinate changes $\bvarsigma_j \colon H^*_T(W) \to H^*(Z)$. These are determined by the following formula (see \cite[Corollary 4.9, Proposition 5.7, (5.21)]{Iritani:monoidal}): 
\[
\scrF_{Z,j}(M_W(\theta) \phi) = \left.\frq^{-\frac{\rho_Z}{(r-1)z}} M_Z(h_{Z,j}+\ovbvarsigma_j(\theta); Q_Z \frq^{-\frac{\rho_Z}{r-1}}) \right|_{Q_Z \to Q} \FT_{Z,j}\sphat(\phi). 
\]
where we write $\bvarsigma_j(\theta) =- (r-1)\lambda_j + h_{Z,j} + \ovbvarsigma_j(\theta)$ and the subscript $Q_Z \to Q$ denotes the replacement of $Q_Z^d$ with $Q^{\imath_*d}$ for $d\in \NEN(Z)$ (where $\imath \colon Z \to X$ is the inclusion). Using the Divisor Equation $M_Z(\tau+h;Q_Z) = e^{h/z}M_Z(\tau,Q_Z e^h)$, which holds for $h\in H^2(Z)$, we can rewrite this as 
\[
\scrF_{Z,j}(M_W(\theta)\phi) = \left. \frt_j^{-\frac{\rho_Z}{(r-1)z}} M_Z(\ovbvarsigma_j(\theta); Q_Z \frt_j^{-\frac{\rho_Z}{r-1}})\right|_{Q_Z \to Q} \FT_{Z,j} \sphat(\phi) 
\]
where $\frt_j := e^{-2\pi \iu j} \frt = e^{-2\pi\iu (j+\frac{1}{2})}\frq$. The statement about $\varsigma_j = \bvarsigma_j \circ (\ttau|_\sfH)^{-1}$ and $\Psi_{Z,j}$ follows from the Lemma \ref{lem:cont_FT} below. 
\end{proof} 

\begin{lemma} 
\label{lem:cont_FT} 
For $\bbf \in H^*_T(W;\Q)$, we have 
\begin{align*} 
\scrF_{Z,0}(\bbf) & \in q_{Z,0} \frt^{-\frac{\rho_Z}{(r-1)z}} H^*(Z;\Q)\otimes \Q[z,z^{-1}](\!(\frt^{-\frac{1}{r-1}})\!) \\ 
\scrF_{Z,j}(\bbf) & = \left. \scrF_{Z,0}(\bbf)\right|_{\frt\to e^{-2\pi\iu j}\frt}. 
\end{align*} 
\end{lemma} 
\begin{proof} 
This follows from the definition of $\scrF_{Z,j}$ in \cite[Section 4.2.3]{Iritani:monoidal}. It suffices to prove the first formula, as the second holds by definition. The equation following \cite[(4.9)]{Iritani:monoidal} applied to $F=Z$ yields: 
\begin{align}
\nonumber  
\scrF_{Z,0}(\bbf) & = \sqrt{\frac{\lambda_0}{c_Z}} (-\lambda_0)^{-\rho_Z/z - r/2} \lambda_0^{-1/2} 
\left[ e^{(z\partial_u)^2}e^{-g(u,\lambda_0)/z} e^{u/\sqrt{c_Z\lambda_0}} \Phi(u)\right]_{u=0} \\ 
\label{eq:scrF0}
& =q_{Z,0} \frt^{-\frac{\rho_Z}{(r-1)z}}
\left[ e^{(z\partial_u)^2}e^{-g(u,\lambda_0)/z} e^{u/\sqrt{c_Z\lambda_0}} \Phi(u)\right]_{u=0} 
\end{align} 
where $\lambda_0 = - \frt^{\frac{1}{r-1}}$, $c_Z = -(r-1)$ and 
\begin{align*} 
g(u,\lambda_0) & = c_Z\lambda_0 \sum_{n=3}^\infty \frac{n-1}{n!} \left(\frac{u}{\sqrt{c_Z\lambda_0}}\right)^n, \\
\Phi(u) & =  e^{-\frac{u}{\sqrt{c_Z\lambda_0}}(\rho_F/z + r_F/2)} \left(\prod_\alpha \tDelta_\alpha^{-1}\right) i_F^* \bbf. 
\end{align*} 
Here $\tDelta_\alpha$ depends on $u$ through the identification $\alpha = w_\alpha \lambda$ and $\lambda = \lambda_0 e^{u/\sqrt{c_Z \lambda_0}}$. By the change of coordinate $v = u/\sqrt{c_Z\lambda_0}$, we see that the quantity in brackets $[\cdots]_{u=0}$ in \eqref{eq:scrF0} is a formal Laurent series in $(z,\frt^{-\frac{1}{r-1}})$ with rational coefficients. 
\end{proof} 

\begin{remark} 
Recall that we considered the base change of the quantum $D$-module of $Z$ via the ring homomorphism \eqref{eq:embed}. 
With this understood, the constant term $h_{Z,j} =\frac{2 \pi \iu}{r-1} (j+\frac{1}{2}) \rho_Z$ for $\varsigma_j(\ttau)$ effectively rescales the Novikov variable $\frq$ by $e^{-2\pi\iu (j+\frac{1}{2})}$, transforming $\frq^{\frac{1}{r-1}}$ into $\frt_j^{\frac{1}{r-1}}=e^{-\frac{2\pi\iu}{r-1} j}\frt^{\frac{1}{r-1}}$. That is, the quantum product $\star^Z_\sigma$ of $Z$ satisfies 
\[
\left. \phantom{*^Z_{h_{Z,j}}}\star^Z_{\varsigma_j(\ttau)}\right|_{Q_Z \to Q \frq^{-\rho_Z/(r-1)}} =\left. \star^Z_{\varsigma_j(\ttau)-h_{Z,j}}\right|_{Q_Z \to Q \frt_j^{-\rho_Z/(r-1)}}.  
\] 
Note that the right-hand side is evidently a (Laurent) power series in $Q$, $\frt_j^{-\frac{1}{r-1}}$, $\ttau$ with rational coefficients. 
\end{remark} 

\begin{remark} 
We have $h_{Z,j} \in \pi \iu H^2(Z;\Q)$ and 
$q_{Z,j}\in \frac{1}{\iu\sqrt{r-1}}\Q(e^{\frac{\pi\iu}{r-1}})\frt^{-\frac{r}{2(r-1)}} \subset 
\Q(e^{\frac{\pi\iu}{2(r-1)}}) \frq^{-\frac{r}{2(r-1)}}$. Here we note that $\iu, \sqrt{r-1} \in \Q(e^{\frac{\pi\iu}{2(r-1)}})$. 
\end{remark} 

\section{Hodge-theoretic properties}

In this section, we prove that the decomposition isomorphism \eqref{eq:Psi} is equivariant with respect to the (universal) Hodge group, as suggested by Katzarkov-Kontsevich-Pantev-Yu \cite{KKPY:birational}. 

First, we recall the Hodge group (also known as the special Mumford-Tate group) associated with a $\Q$-Hodge structure (see e.g.~\cite{GGK:MTgroups}). Let $H$ be a $\Q$-Hodge structure; that is, a finite-dimensional rational vector space $H$ equipped with a decomposition $H_\C = \bigoplus_{p,q\in \Z} H^{p,q}$ such that $H^{q,p} = \overline{H^{p,q}}$, where $H_\C = H\otimes \C$. Let $S^1 = \{z\in \C : |z|^2=1\} \cong \{(x,y) \in \R^2 : x^2 + y^2 = 1\}$ be the circle group, viewed as an algebraic group defined over $\R$. We consider the $S^1$-action on $H_\C$ given by $\lambda \cdot u^{p,q} = \lambda^{p-q} u^{p,q}$ for $\lambda \in S^1$ and $u^{p,q} \in H^{p,q}$. This action preserves the real form $H_\R = H \otimes \R$, inducing a homomorphism $h \colon S^1 \to \GL(H_\R)$.  The \emph{Hodge group} of $H$ is defined to be the smallest algebraic subgroup $\Hod_H \subset \GL(H)$, defined over $\Q$, such that the set of real points $\Hod_H(\R)$ contains the image $h(S^1)$. Note that a Tate twist of the Hodge structure does not change this group. When $H$ is the rational cohomology of a smooth projective variety, the $\Hod_H$-fixed subspace $H^{\Hod_H} = H\cap (\bigoplus_p H^{p,p})$ consists of rational Hodge classes. We refer to elements of $H_\C^{\Hod_{H,\C}} = H^{\Hod_H}\otimes \C$ as \emph{complexified Hodge classes}. 

The Hodge groups have the following Tannakian interpretation. Consider the category $(\operatorname{Hodge}_\Q)$ of polarizable $\Q$-Hodge structures and let $\cC=(\operatorname{Hodge}_\Q)/\operatorname{Tate}$ be its orbit category by Tate twists. The category $\cC$ is a neutral Tannakian category with fibre functor $\omega\colon \cC\to (\operatorname{Vect}_\Q)$ sending a Hodge structure to its underlying $\Q$-vector space. The \emph{universal Hodge group} is the affine group scheme $\Hod =\Aut^{\otimes}(\omega)$ associated with $\cC$ in the Tannakian formalism \cite[Chapter II]{DMOS:Hodge}. The universal Hodge group $\Hod$ can be described as the inverse limit\footnote{Here we regard $H\mapsto \Hod_H$ as an inverse system via inclusions of Hodge structures; any inclusion $H_1 \hookrightarrow H_2$ of Hodge structures induces a surjective group homomorphism $\Hod_{H_2}\twoheadrightarrow \Hod_{H_1}$.} $\varprojlim_H \Hod_H$ of Hodge groups over all polarizable $\Q$-Hodge structures $H$. 
It acts on each object $H$ of $\cC$ via the canonical surjection $\Hod \twoheadrightarrow \Hod_H\subset \GL(H)$ and every morphism in $\cC$ is $\Hod$-equivariant.


\begin{lemma} 
\label{lem:quantum_product} 
The big quantum product of a smooth projective variety $Y$ is equivariant with respect to the universal Hodge group $\Hod$, i.e. we have 
\[
g(\alpha \star_t \beta) = g(\alpha) \star_{g(t)} g(\beta) 
\]
for $g\in \Hod$, where $\alpha,\beta,t\in H^*(Y;\Q)$ and $\star_t$ denotes the big quantum product of $Y$ with parameter $t$. 
\end{lemma}

\begin{proof} 
We note that a stronger result, the equivariance with respect to the motivic Galois group of Andr\'e motives, is discussed in \cite[Proposition 3.40]{KKPY:birational}. 

It suffices to show that the quantum product is equivariant with respect to the Hodge group $\Hod_Y := \Hod_{H^*(Y;\Q)}$ associated with the rational cohomology of $Y$. Recall that the big quantum product for $Y$ is defined by the formula: 
\[
(\alpha\star_t \beta,\gamma) = \sum_{d\in \NEN(X)} \sum_{n=0}^\infty \corr{\alpha,\beta,\gamma,t,\dots,t}_{0,3+n,d}^Y \frac{Q^d}{n!}.  
\]
Thus, it suffices to show that the Gromov-Witten correlators  $\corr{\cdots}_{0,n,d}^Y$ and the Poincar\'e pairing $(\cdot,\cdot)$ are $\Hod_Y$-invariant. They are multilinear forms defined over $\Q$. The algebraicity of the (virtual) fundamental class implies that they are invariant under $h\colon S^1 \to \GL(H^*(Y;\R))$. The algebraic subgroup of $\GL(H^*(Y;\Q))$ preserving the form $\corr{\cdots}_{0,n,d}^Y$ or $(\cdot,\cdot)$ is defined over $\Q$ and the set of its real points contains $h(S^1)$. Hence, it must contain $\Hod_Y$. 
\end{proof} 

By Lemma \ref{lem:quantum_product}, we may regard the quantum $D$-module of $Y$ as a $\Hod_\C$-equivariant vector bundle equipped with a $\Hod_\C$-invariant flat connection, where $\Hod_\C$ acts on the parameter $t\in H^*(Y)$ and the cohomology fiber via the projection $\Hod_\C \to \Hod_{Y,\C}$, and acts trivially on $z$ and the Novikov variables  $Q,\tQ,Q_Z$. 

\begin{proposition} 
\label{prop:Hodge-equivariance} 
The formal maps $\tau \colon H^*(\tX)\to H^*(X)$, $\varsigma_j \colon H^*(\tX) \to H^*(Z)$ and the isomorphism $\Psi$ appearing in \eqref{eq:Psi} are $\Hod_\C$-equivariant. 
\end{proposition} 

\begin{remark} 
We have an isomorphism of Hodge structures \cite[Theorem 7.31]{Voisin:Hodge_I}
\begin{equation} 
\label{eq:dec} 
\dec \colon H^*(X;\Q) \oplus \bigoplus_{j=0}^{r-2} H^*(Z;\Q)(-j-1) \cong H^*(\tX) 
\end{equation} 
defined by $\dec(\alpha,\beta_0,\dots,\beta_{r-2}) = \varphi^*\alpha + \sum_{j=0}^{r-2} \jmath_*(p^j \pi^* \beta_j)$. This is $\Hod$-equivariant. Therefore, the universal Hodge group $\Hod$ acts on the cohomologies of $X$, $Z$ and $\tX$ through the algebraic quotient $\Hod_{\tX} = \Hod_{H^*(\tX;\Q)}$. 
\end{remark} 

\begin{proof}[Proof of Proposition \ref{prop:Hodge-equivariance}] 
The proposition can be proved directly from the construction in \cite{Iritani:monoidal} (as we did in the proof of Proposition \ref{prop:cyclotomic}) or by using the reconstruction method of Hinault-Yu-Zhang-Zhang \cite{HYZZ:framing}. Here, we adopt the latter approach, which does not require the full details of the previous proof and is, hence, more transparent to the reader. 

We review the reconstruction of the decomposition $\Psi$ via the initial conditions and the Birkhoff factorization described in \cite[Section 5.8]{Iritani:monoidal}. To simplify notation, we write $H_{\rm decomp}:=H^*(X) \oplus H^*(Z)^{\oplus (r-1)}$. We start with the initial conditions 
\begin{align*} 
& \tau^\circ \in H^*(X) \otimes \C(\!(\frq^{-1})\!), \quad \varsigma_j^\circ \in H^*(Z)\otimes \C(\!(\frq^{-\frac{1}{r-1}})\!), \\
& \Psi^\circ \in \Hom(H^*(\tX), H_{\rm decomp}) \otimes \C[z](\!(\frq^{-1/\frs})\!).  
\end{align*} 
These quantities are given explicitly in \cite[Section 5.8]{Iritani:monoidal}. They are the restrictions, respectively, of $\tau(\ttau)$, $\varsigma_j(\ttau)$ and $\Psi$ to the locus $Q= \ttau=0$. 
We introduce a coordinate system $(t,s_0,\dots,s_{r-2})$ on $H^*(X)\oplus H^*(Z)^{\oplus (r-1)}$, which is related to the original one,  $(\tau, \varsigma_0,\dots,\varsigma_{r-2})$, by $\tau = \tau^\circ + t$ and $\varsigma_j = \varsigma_j^\circ + s_j$. We also relate it to $\ttau \in H^*(\tX)$ via the formal change of variables $\tau=\tau(\ttau)$ and $\varsigma_j = \varsigma_j(\ttau)$. 
Let $M\in \End(H_{\rm decomp})\otimes \C[z^{-1}](\!(\frq^{-1/\frs})\!)[\![Q,t,s]\!]$ be the following block-diagonal endomorphism: 
\[
M = \begin{pmatrix} 
e^{-\tau^\circ/z} M_X(\tau^\circ+t) & & & \\ 
& e^{-\varsigma_0^\circ/z} M_Z(\varsigma_0^\circ+s_0) & & \\
& & \ddots & \\
& & & e^{-\varsigma_{r-2}^\circ/z} M_Z(\varsigma_{r-2}^\circ+s_{r-2}) 
\end{pmatrix},
\]
where $M_X$ and $M_Z$ are the fundamental solutions \cite[(2.5)]{Iritani:monoidal} for the quantum connections of $X$ and $Z$, respectively, with the base change to $\C(\!(\frq^{-1/\frs})\!)[\![Q]\!]$ via \eqref{eq:embed} understood. We have $M|_{Q=t=s=0}=\id$. The isomorphism $\Psi$ as a function of $(t,s_0,\dots,s_{r-2})$ is uniquely determined by the equation 
\[
(\Psi^\circ)^{-1} \circ M =M' \circ \Psi^{-1},
\]
where 
\begin{align*} 
& M'\in \End(H^*(\tX))\otimes \C[z^{-1}](\!(\frq^{-1/\frs})\!)[\![Q,t,s]\!] \quad \text{satisfying $M'|_{z=\infty}=\id$,} \\
& \Psi\in \Hom(H^*(\tX),H_{\rm decomp})\otimes \C[z](\!(\frq^{-1/\frs})\!)[\![Q,t,s]\!] \quad \text{satisfying $\Psi|_{Q=t=s=0}=\Psi^\circ$.} 
\end{align*} 
Regarding $z$ as a loop parameter, we find that $M'$ and $\Psi$ are, respectively, the negative and positive Birkhoff factors of $(\Psi^\circ)^{-1} \circ M$. As explained in \cite[Section 5.8]{Iritani:monoidal}, $M'$ gives a fundamental solution for $\tX$ in the $(Q,t,s)$-direction. The initial condition $M'|_{Q=t=s=0}=\id$ implies that $M' = (M_\tX(\ttau)|_{Q=t=s=0})^{-1} \circ M_\tX(\ttau)$. The coordinate change between $\ttau$ and $(t,s_0,\dots,s_{r-2})$ is given by the asymptotics 
\begin{equation}
\label{eq:ttau_asymptotics} 
M' 1 = (M_\tX(\ttau)|_{Q=t=s=0})^{-1} \circ M_\tX(\ttau) 1 = 1 + \ttau z^{-1} + O(z^{-2}). 
\end{equation} 
where we used the fact that $M_\tX(\ttau) 1 = 1 +\ttau z^{-1} + O(z^{-2})$, $M_\tX(\ttau)^{-1} 1 = 1 - \ttau z^{-1} + O(z^{-2})$ and $\ttau|_{Q=t=s=0} = 0$. 
This determines $\ttau$ as a function of $(t,s_0,\dots,s_{r-2})$, and conversely, $(t,s_0,\dots,s_{r-2})$ (and consequently $(\tau,\varsigma_0,\dots,\varsigma_{r-2})$) as a function of $\ttau$. 

We proceed to the proof of the proposition. As we will see in Lemma \ref{lem:invariance_initial} below, $\tau^\circ$ and  $\varsigma_j^\circ$ are complexified Hodge classes (i.e.~fixed by $\Hod_\C$) and $\Psi^\circ$ is $\Hod_\C$-equivariant. The block-diagonal endomorphism $M=M(t,s_0,\dots,s_{r-2})$ is $\Hod_\C$-equivariant in the sense that 
\[
g M(t,s_0,\dots,s_{r-2}) g^{-1} = M(g(t), g(s_0),\dots,g(s_{r-2})) \qquad \text{for $g\in \Hod_\C$.}  
\]
This follows from the corresponding property $g M_Y(\ttt) g^{-1} = M_Y(g(\ttt))$ for the fundamental solution $M_Y$; since $M_Y$ is defined in terms of descendant Gromov-Witten invariants, we can deduce the $\Hod_\C$-equivariance of $M_Y$ by an argument similar to that in the proof of Lemma \ref{lem:quantum_product}. Hence, the composition $(\Psi^\circ)^{-1} \circ M$ satisfies the same $\Hod_\C$-equivariance. It is straightforward to prove that the Birkhoff factors $M'$ and $\Psi^{-1}$ inherit this  $\Hod_\C$-equivariance. In particular, the map $(t,s_0,\dots,s_{r-2})\mapsto \ttau$ determined by \eqref{eq:ttau_asymptotics} is also $\Hod_\C$-equivariant. 
\end{proof} 

\begin{lemma} 
\label{lem:invariance_initial} 
The initial conditions $\tau^\circ$ and $\varsigma_j^\circ$ are fixed by $\Hod_\C$ and $\Psi^\circ$ is $\Hod_\C$-equivariant. 
\end{lemma} 
\begin{proof} 
We summarize the initial conditions $\tau^\circ$, $\varsigma_j^\circ$ and $\Psi^\circ$ given in \cite[Section 5.8.1]{Iritani:monoidal}. First, $\tau^\circ$ and $\varsigma_j^\circ$ are given by 
\begin{align*}
\begin{split} 
\tau^\circ & = [z^{-1}] \log \left( 1 + \sum_{k>0} \frq^{-k}
\frac{\imath_*(\prod_{\nu=1}^{k-1} e_{-\nu z}(\cN_{Z/X}))}{k!z^k}  \right), \\
\varsigma_j^\circ & = -(r-1) \lambda_j + [z^{-1}] \log  \left( \frq^{\rho_Z/((r-1)z)} \scrF_{Z,j}(1) \right),  
\end{split} 
\end{align*} 
where $e_\lambda(\cdots)$ is the equivariant Euler class \cite[(2.21)]{Iritani:monoidal} and $[z^{-1}](\cdots)$ denotes the coefficient of $z^{-1}$.  
We define a map $\dec_T\colon H_{\rm decomp} \to H^*_T(W)$ by 
\[
\dec_T(\alpha,\beta_0,\dots,\beta_{r-2}) = \hvarphi^*\pr_1^*\alpha + \sum_{j=0}^{r-2} \hjmath_*(\hp^j \hpi^*\beta_j), 
\]
where $\pr_1\colon X\times \PP^1 \to X$ is the first projection, $\hvarphi\colon W\to X\times \PP^1$ is the blowup along $Z\times \{0\}$ and the maps $\hjmath\colon \hD\to W$, $\hpi\colon \hD\to Z$ and the class $\hp\in H^2(\hD;\Z)$ are as in \cite[Section 3.5]{Iritani:monoidal}. This is a lift of $\dec$ in \eqref{eq:dec} and satisfies $\kappa_\tX \circ \dec_T = \dec$ for the Kirwan map $\kappa_\tX \colon H^*_T(W) \to H^*(\tX)$ (see \cite[Section 3.6]{Iritani:monoidal}). Let $\kappa_\tX^{-1}$ denote the right inverse of $\kappa_\tX$ given by 
\[
\kappa_\tX^{-1} := \dec_T \circ \dec^{-1} \colon H^*(\tX) \to H^*_T(W). 
\]
These maps $\dec$, $\dec_T$ and $\kappa_\tX^{-1}$ are $\Hod$-equivariant as they are morphisms in $\cC=(\operatorname{Hodge}_\Q)/\operatorname{Tate}$. 
We write $\Psi^\circ =(\Psi_X^\circ,\Psi_{Z,0}^\circ,\dots,\Psi_{Z,r-2}^\circ)$ as before, where $\Psi_X^\circ$ is the $\Hom(H^*(\tX),H^*(X))$-component and $\Psi_{Z,j}^\circ$ is the $j$-th $\Hom(H^*(\tX),H^*(Z))$-component. The map $\Psi^\circ$ is given by 
\begin{align*} 
\Psi_X^\circ (\gamma) & =  e^{-\tau^\circ/z}\left(\kappa_X(\kappa_\tX^{-1}\gamma) + \sum_{k>0} \frq^{-k}
\imath_*\left( \frac{\prod_{\nu=1}^{k-1} e_{-\nu z}(\cN_{Z/X})}{k!z^k}  \left[i_Z^*\kappa_\tX^{-1}\gamma \right]_{\lambda=kz} \right)\right) \\
\Psi_{Z,j}^\circ(\gamma) & = e^{-(\varsigma_j^\circ+(r-1)\lambda_j)/z} \frq^{\rho_Z/((r-1)z)} \scrF_{Z,j}(\kappa_\tX^{-1}\gamma)
\end{align*} 
for $\gamma \in H^*(\tX)$, where $i_Z \colon Z\to W$ is the inclusion and $\kappa_X \colon H^*_T(W) \to H^*(X)$ is the Kirwan map. 

We remark that $H_T^*(W;\Q)$ has a canonical polarizable $\Q$-Hodge structure, and hence is a $\Hod$-module. For a fixed degree $k$, $H^k_T(W;\Q)$ can be identified with the cohomology of a (sufficiently large) finite-dimensional approximation of the Borel construction, which is a smooth projective variety, and hence admits a polarizable pure $\Q$-Hodge structure of weight $k$. The equivariant parameter $\lambda \in H^2_T(\pt;\Z)$ acts on $H_T^*(W;\Q)$ as an operator of type $(1,1)$. 

By the construction of $\tau^\circ$, $\varsigma_j^\circ$ and $\Psi^\circ$ above, it suffices to prove that the following maps are $\Hod_\C$-equivariant: 
\begin{enumerate} 
\item multiplication by $e_{-\nu z}(\cN_{Z/X})$ in $H^*(Z)\otimes \C[z,z^{-1}]$; 
\item $\imath_* \colon H^*(Z) \to H^*(X)$; 
\item $i_Z^* \colon H^*_T(W) \to H^*(Z)\otimes \C[\lambda]$; 
\item $\kappa_X \colon H^*_T(W) \to H^*(X)$; and 
\item $\frq^{\rho_Z/((r-1)z)}\scrF_{Z,j} \colon H^*_T(W) \to \frq^{-\frac{r}{2(r-1)}}H^*(Z)\otimes \C[z,z^{-1}](\!(\frq^{-\frac{1}{r-1}})\!)$. 
\end{enumerate} 
Part (1) is obvious as $e_{-\nu z}(\cN_{Z/X})$ is an algebraic class. Parts (2)-(4) follow from the fact that the maps $\imath_*$, $i_Z^*$ and $\kappa_X$ are morphisms in  $\cC$. Note that the Kirwan map $\kappa_X$ is defined to be the composition 
\[
H^*_T(W) \xrightarrow{i^*} H^*_T(W_{\st}) \xrightarrow[\cong]{(\pr^*)^{-1}} H^*(W_{\st}/T) = H^*(X),
\]
in which all maps are morphisms of mixed Hodge structures, where $i\colon W_\st \to W$ is the inclusion of the open stable locus and $\pr \colon W_\st \to W_\st/T=X$ is the projection. (In the case at hand, $\kappa_X$ can also be described as the composition $H^*_T(W) \xrightarrow{i_X^*} H^*_T(X) = H^*(X)[\lambda] \xrightarrow{\lambda \to 0} H^*(X)$.) The map $\frq^{\rho_Z/((r-1)z)}\scrF_{Z,j}$ only involves  multiplication by algebraic classes and $i_F^*$ as can be seen from \eqref{eq:scrF0}, and hence is $\Hod_\C$-equivariant. Note here that the quantum Riemann-Roch operator $\tDelta_\alpha$ involves only multiplication by algebraic classes $\ch_i(\cN_{Z/W,\alpha})$. 
\end{proof} 

\begin{corollary} 
\label{cor:Hodge-fixed}
The formal map $(\tau,\varsigma_0,\dots,\varsigma_{r-2}) \colon H^*(\tX) \to H^*(X)\oplus H^*(Z)^{\oplus (r-1)}$ restricts to a formal isomorphism between the subspaces of complexified Hodge classes.  
When $\ttau$ is a complexified Hodge class, the decomposition isomorphism $\Psi|_\ttau$ restricts to an isomorphism between these Hodge subspaces. 
\end{corollary} 

\begin{remark} 
By a similar argument, we can show that the formal maps $\tau\colon H^*(\tX) \to H^*(X)$ and  $\varsigma_j\colon H^*(\tX) \to H^*(Z)$ preserve (complexified) algebraic classes, i.e.~complex linear combinations of the Poincar\'e duals of algebraic cycles, and that the isomorphism $\Psi|_\ttau$ preserves algebraic classes whenever the parameter $\ttau$ is algebraic. By the algebraic construction of virtual fundamental classes, we have a quantum product defined within Chow groups, and the fundamental solution $M_Y(t)$ preserves algebraic classes whenever the parameter $t$ is algebraic. 
\end{remark}


\bibliographystyle{amsplain}
\bibliography{notes_decomposition}

\end{document}